\documentclass[11pt,reqno]{amsart}

\numberwithin{equation}{section}

\usepackage{color,graphics,epsfig,amsmath,a4wide,psfrag}

\newcommand{\mbfc}{\mathbf{c}}
\newcommand{\mbfd}{\mathbf{d}}

\newcommand{\mbfg}{\mathbf{g}}

\newcommand{\mbfn}{\mathbf{n}}

\newcommand{\mbfp}{\mathbf{p}}

\newcommand{\mbfx}{\mathbf{x}}
\newcommand{\mbfy}{\mathbf{y}}

\newcommand{\mC}{\mathbb{C}}
\newcommand{\mD}{\mathbb{D}}

\newcommand{\mF}{\mathbb{F}}

\newcommand{\mR}{\mathbb{R}}
\newcommand{\mS}{\mathbb{S}}

\newtheorem{theorem}{Theorem}[section]
\newtheorem{lemma}[theorem]{Lemma}

\newtheorem{proposition}[theorem]{Proposition}

\theoremstyle{definition}

\theoremstyle{definition}

\theoremstyle{definition}

\begin{document}



\title[Generalized chordal metric]{A generalized chordal metric making strong stabilizability a robust property}

\author{Amol Sasane}
\address{Department of Mathematics, London School of Economics,
     Houghton Street, London WC2A 2AE, United Kingdom.}
\email{sasane@lse.ac.uk}

\begin{abstract}
 An abstract chordal metric is defined  on linear control systems 
 described by their transfer functions. Analogous to a previous result due to 
 Jonathan~Partington~\cite{Par} for $H^\infty$, it is shown  
 that strong stabilizability is a robust property in this metric. 
\end{abstract}

\maketitle

\section{Introduction}

The aim of this note is to give an extension of a 
result due to Jonathan Partington (recalled below in Proposition~\ref{theorem_partington}) saying that 
 {\em strong} stabilizability is a robust property of the plant in the chordal metric.  
The basic and almost unique ingredient in the proof of this fact is a result proved by 
Partington in \cite[Lemma~2.1, p.84]{Par92} (which we have restated in 
Lemma~\ref{lemma_partington}). The only new point is that we prove that 
the analogous result holds in an abstract setting, hence expanding the domain of applicability from the original setting 
of unstable plants over $H^\infty$ to ones over arbitrary rings of stable 
transfer functions satisfying mild assumptions. (Here, as is usual in the control engineering literature, $H^\infty$ denotes the Hardy algebra of 
bounded holomorphic functions defined in the complex open right half plane $\{s\in \mC: \textrm{Re}(s)>0\}$.)

We recall the general {\em stabilization problem} in control theory.
Suppose that $R$ is an integral domain with identity
(thought of as the class of stable transfer functions) and let
$\mF(R)$ denote the field of fractions of $R$. Then the stabilization
problem is:

\begin{itemize}
 \item[] 
Given $\mbfp\in \mF(R) $ (an unstable plant transfer function),

 \item[] find $\mbfc \in \mF(R)$ (a stabilizing controller
  transfer function),

 \item[] such that (the closed loop transfer function)
$$
H(\mbfp,\mbfc):= \left[\begin{array}{cc} {\displaystyle \frac{\mbfp}{1-\mbfp\mbfc} }_{\phantom{p}} & \displaystyle\frac{\mbfp\mbfc}{1-\mbfp\mbfc} \\
                        \displaystyle\frac{\mbfp\mbfc}{1-\mbfp\mbfc} & \displaystyle\frac{\mbfc}{1-\mbfp\mbfc} 
                       \end{array}\right]
$$

\item[] belongs to $R^{2\times 2}$ (that is, it is stable). 
\end{itemize}

\medskip

\noindent The demand above that
$H(\mbfp,\mbfc)\in R^{2\times 2}$ guarantees that the ``closed loop'' transfer function of the
signal map 
$$
\left[ \begin{array}{cc}
u_1\\
u_2 \end{array}\right]
\mapsto
\left[ \begin{array}{cc}
y_1 \\
y_2 \end{array}\right],
$$
in the interconnection of $\mbfp$ and $\mbfc$ as shown in Figure~\ref{figure_feedback_config}, is stable. (So after the interconnection, ``nice'' signals are indeed
mapped to nice signals.) 
\begin{figure}[h]
    \center
    \psfrag{P}[c][c]{$\mbfp$}
    \psfrag{C}[c][c]{$\mbfc$}
    \psfrag{u}[c][c]{$u_1$}
    \psfrag{v}[c][c]{$u_2$}
    \psfrag{y}[c][c]{$y_1$}
    \psfrag{z}[c][c]{$y_2$}
    \includegraphics[width=5.4 cm]{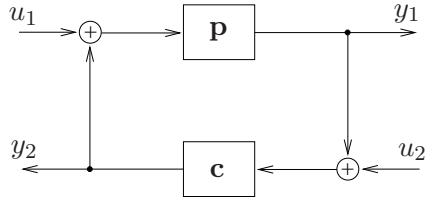}
    \caption{Feedback connection of the plant $\mbfp$ with the controller $\mbfc$.}
    \label{figure_feedback_config}
 \end{figure}

A stronger version of the problem is when we require a 
{\em stable} controller $\mbfc\in R$ which stabilizes $\mbfp$. If such a $\mbfc$ exists, then 
we say that $\mbfp$ is {\em strongly stabilizable}.

 In the {\em robust stabilization problem}, one goes a step
further than the stabilization problem.  
One knows that the plant is just an approximation of
reality, and so one would really like the controller $\mbfc$ to not only
stabilize the {\em nominal} plant $\mbfp_0$, but also all sufficiently close
plants $\mbfp$ to $\mbfp_0$.  The question of what one means by ``closeness''
of plants thus arises naturally. So one needs a function $d$ defined
on pairs of stabilizable plants such that
\begin{enumerate}
\item $d$ is a metric on the set of all stabilizable
plants,
\item $d$ is amenable to computation, and
\item stabilizability is a robust property of the plant with respect
  to $d$.
\end{enumerate}
There are various known metrics which do the job, notably the gap metric (\cite{ZamElS}), the graph metric (\cite{Vid}) 
and the Vinnicombe $\nu$-metric (see \cite{Vin} for the rational transfer function case and \cite{BalSas}, \cite{Sas} 
for its recent extension for nonrational transfer functions). This last metric is in some sense the ``best'' one, 
as it is comparatively easy to compute and admits some sharp robustness results. The Vinnicombe metric 
itself arose from a very natural idea of defining a metric between meromorphic functions in the 
complex right half plane, namely the pointwise chordal metric, defined below. This metric has been studied by 
function theorists (see for example \cite{Hay}), since it is a natural analogue of the $H^\infty$ 
distance between bounded analytic functions, and it can be used for functions with poles in a disk. 
The use of the chordal metric to study robustness of stabilizability was made by Ahmed El-Sakkary in \cite{Sak}.

If $\mbfp_1,\mbfp_2$ are two meromorphic functions in the open right half plane, then the 
{\em chordal distance} $\kappa$ between $\mbfp_1,\mbfp_2$ is 
$$
\kappa(\mbfp_1,\mbfp_2):=\sup_{\substack{s\in \mC;\; \textrm{Re}(s)>0;\\ \textrm{either } 
\mbfp_1(s)\neq \infty\textrm{ or } \mbfp_2(s)\neq \infty}} \frac{|\mbfp_1(s)-\mbfp_2(s)|}{\sqrt{1+|\mbfp_1(s)|^2} \sqrt{1+|\mbfp_2(s)|^2}}.
$$
This metric has the interpretation that it is the supremum of the pointwise 
Euclidean distance between the points $\mbfp_1(s)$ and $\mbfp_2(s)$ on the Riemann sphere. 
Recall that the stereographic projection allows the identification of  the extended 
complex plane $\mC\cup \{\infty\}$ with the unit sphere $\mathbf{S}$ of diameter $1$ 
in $\mR^3$, where the point $z=0$ in the complex plane corresponds 
to the south pole $S$ of the sphere $\mathbf{S}$ and the point $z=\infty$ 
corresponds to the north pole $N$ of $\mathbf{S}$. Points $P_\mC$ in the complex plane 
can be identified with a corresponding point $P_{\mathbf{S}}$ on the sphere $\mathbf{S}$, 
namely the one in $\mathbf{S}$ which lies on the straight line joining $P_\mC$ and $N$. See Figure~\ref{riemann}.  
\begin{figure}[h]
    \center
    \psfrag{P}[c][c]{${\scriptscriptstyle P_{\scriptscriptstyle\mathbf{S}}}$}
    \psfrag{z}[c][c]{${\scriptscriptstyle P_{\scriptscriptstyle\mC}\equiv z}$}
    \psfrag{S}[c][c]{${\scriptscriptstyle \textrm{S}\equiv 0}$}
    \psfrag{N}[c][c]{${\scriptscriptstyle \textrm{N}\equiv \infty}$}
    \psfrag{C}[c][c]{$\mC$}
    \psfrag{s}[c][c]{${\mathbf{S}}$}   
    \psfrag{n}[c][c]{${\scriptscriptstyle \frac{1}{2}}$}
    \includegraphics[width=9 cm]{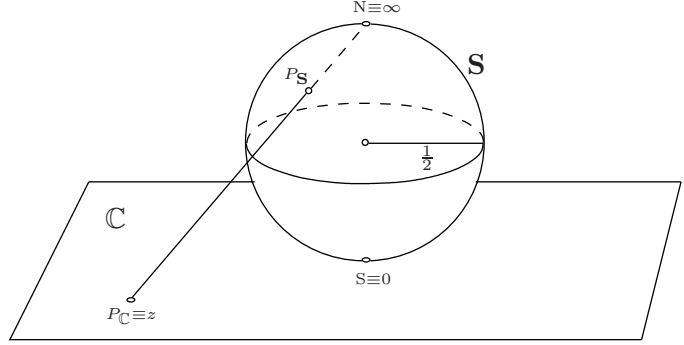}
    \label{riemann}
    \caption{The Riemann sphere with diameter $1$ and centre at $\left(0,0,\displaystyle \frac{1}{2}\right)$.}
 \end{figure}

The following result was shown by Jonathan Partington (see \cite[Theorem~2.2, p.84]{Par92} or \cite[Theorem~4.3.4, p.83]{Par}). 

\begin{proposition}
 \label{theorem_partington} 
Let $\mbfp_0, \mbfp \in \mF(H^\infty)$, and let $\mbfc\in H^\infty$ be such that 
$
\mbfg_0:= \displaystyle \frac{\mbfp_0}{1-\mbfc\mbfp_0} \in H^\infty.
$
Set $k:= \|\mbfc\|_\infty$ and $g=\|\mbfg_0\|_\infty$. If 
$$
\kappa(\mbfp,\mbfp_0)< \displaystyle \frac{1}{3} \min\left\{1, \;\;\frac{1}{g}, \;\;\frac{1}{k(1+kg)} \right\},
$$
then $\mbfp$ is also stabilized by $\mbfc$.
\end{proposition}

This  follows from the following key estimate, which gives a lower bound on the chordal distance; 
see \cite[Lemma~2.1, p.84]{Par92} or \cite[Lemma~4.3.3]{Par}. 

\begin{lemma}
 \label{lemma_partington}
If $z_1,z_2\in \mC$ and $0<a<1$, then 
$$
\kappa(z_1,z_2) :=\frac{|z_1-z_2|}{\sqrt{1+|z_1|^2} \sqrt{1+|z_2|^2}}\geq 
\min \left\{ \frac{a^2}{1+a^2} |z_1-z_2|, \;\;\frac{a^2}{1+a^2} \left|\frac{1}{z_1}-\frac{1}{z_2} \right|,\;\; \frac{1-a^2}{1+a^2}\right\}.
$$
\end{lemma}

\subsection{Abstract set-up and main result} 

\noindent Our main result is given in  Theorem~\ref{main_theorem} below. We will assume throughout the following:
\begin{itemize}
 \item[(A1)] $R$ is a commutative ring without zero divisors and with identity. 
 \item[(A2)] $S$ is a complex, commutative, unital, semisimple Banach algebra. 
 \item[(A3)] $R\subset S$, that is, there is an injective  ring homomorphism  $\iota: R\rightarrow S$.
 \item[(A4)] $R$ is a full in $S$, that is, if $\mbfx\in R$ and $\iota (\mbfx)$ is invertible in $S$, then $\mbfx$ is invertible in $R$.
\end{itemize}

\noindent (A3) allows identification of 
elements of $R$ with elements of $S$. So in the sequel, if $\mbfx$ is an element of $R$, we will simply write $\mbfx$ (an element of  $S$!) instead of $\iota(\mbfx)$.

We will denote by $\mF(R)$ the field of fractions over $R$. An element $\mbfp\in \mF(R)$ is said to have a {\em  coprime factorization over $R$} if 
$$
\mbfp=\displaystyle \frac{\mbfn}{\mbfd},
$$
where $\mbfn,\mbfd\in R$, $\mbfd\neq 0$ and there exist $\mbfx,\mbfy\in R$ such that $\mbfn\mbfx+\mbfd\mbfy=1$.

We define the subset of {\em coprime factorizable plants over} $R$ to be the set 
$$
\mS(R):=\{\mbfp \in \mF(R): \mbfp \textrm{ has a coprime factorization}\}.
$$
The maximal ideal space of $S$ is denoted by $M(S)$. 
If $\mbfx\in S$, then we denote by $\widehat{\mbfx}$ the Gelfand transform of $\mbfx$. Also, we set 
$$
\|\mbfx\|_\infty:=\max_{\varphi \in M(S)}|\widehat{\mbfx}(\varphi)|.
$$
If $\mbfp_1,\mbfp_2 \in \mS(R)$, then the 
{\em chordal distance} $\kappa$ between $\mbfp_1,\mbfp_2$, which have coprime factorizations 
$$
\mbfp_1=\displaystyle \frac{\mbfn_1}{\mbfd_1}\textrm{ and }\mbfp_2=\displaystyle \frac{\mbfn_2}{\mbfd_2}, 
$$
is 
$$
\kappa(\mbfp_1,\mbfp_2):=\sup_{\varphi \in M(S)} 
\frac{|\widehat{\mbfn_1}(\varphi)\widehat{\mbfd_2}(\varphi)-\widehat{\mbfn_2}(\varphi)\widehat{\mbfd_1}(\varphi)|}{
\sqrt{|\widehat{\mbfn_1}(\varphi)|^2+|\widehat{\mbfd_1}(\varphi)|^2} 
\sqrt{|\widehat{\mbfn_2}(\varphi)|^2+|\widehat{\mbfd_2}(\varphi)|^2}}.
$$
The function $\kappa$ given by the above expression is well-defined. Indeed, if 
$$
\mbfp_1=\frac{\mbfn_1}{\mbfd_1}=\frac{\widetilde{\mbfn}_1}{\widetilde{\mbfd}_1},
$$
then $\mbfn_1 \widetilde{\mbfd}_1=\widetilde{\mbfn}_1\mbfd_1$, and so, for each $\varphi \in M(S)$, we have 
$
\widehat{\mbfn_1}(\varphi) \widehat{\widetilde{\mbfd}_1}(\varphi)=\widehat{\widetilde{\mbfn}_1}(\varphi)\widehat{\mbfd_1}(\varphi).
$ Using this one can see that 
$$
\frac{|\widehat{\mbfn_1}(\varphi)\widehat{\mbfd_2}(\varphi)-\widehat{\mbfn_2}(\varphi)\widehat{\mbfd_1}(\varphi)|}{
\sqrt{|\widehat{\mbfn_1}(\varphi)|^2+|\widehat{\mbfd_1}(\varphi)|^2} }
=
\frac{|\widehat{\widetilde{\mbfn}_1}(\varphi)\widehat{\mbfd_2}(\varphi)-\widehat{\mbfn_2}(\varphi)\widehat{\widetilde{\mbfd}_1}(\varphi)|}{
\sqrt{|\widehat{\widetilde{\mbfn}_1}(\varphi)|^2+|\widehat{\widetilde{\mbfd}_1}(\varphi)|^2} },
 $$
and so it follows that the expression in the definition of $\kappa$ is independent of any particular 
choice of a coprime factorization of either plant.

We have the following result.

\begin{proposition}
\label{main_proposition} 
$\kappa$ is a metric on $\mS(R)$.
\end{proposition}
\begin{proof} The proof is straightforward, but we give the details as they elucidate the 
use of the basic assumptions in our abstract setting. 

\medskip 

\noindent (D1) If $\mbfp_1,\mbfp_2 \in \mS(R)$, then it is clear from the expression for  $\kappa(\mbfp_1,\mbfp_2)$ that it is nonnegative. 
Furthermore, $\kappa(\mbfp,\mbfp)=0$ for any $\mbfp\in \mS(R)$. 

\noindent Finally, if $\mbfp_1, \mbfp_2\in \mS(R)$ are such that 
$\kappa(\mbfp_1,\mbfp_2)=0$, then we must have, 
with $\mbfp_1,\mbfp_2$ having coprime factorizations 
$$
\mbfp_1=\displaystyle \frac{\mbfn_1}{\mbfd_1}\textrm{ and }
\mbfp_2=\displaystyle \frac{\mbfn_2}{\mbfd_2},
$$
that   for all $\varphi \in M(S)$ that 
$
\widehat{\mbfn_1}(\varphi)\widehat{\mbfd_2}(\varphi)-\widehat{\mbfn_2}(\varphi)\widehat{\mbfd_1}(\varphi)=0,
$
and by (A3) and the semisimplicity of the Banach algebra (A2), we obtain  
$\mbfn_1\mbfd_2=\mbfn_2\mbfd_1$, that is, $\mbfp_1=\mbfp_2$. 

\medskip 

\noindent (D2) If $\mbfp_1,\mbfp_2 \in \mS(R)$, then it is clear from the expression for  $\kappa$ that 
 $\kappa(\mbfp_1,\mbfp_2)=\kappa(\mbfp_2,\mbfp_1)$. 

\medskip 

\noindent (D3) Let $\mbfp_1, \mbfp_2, \mbfp_3\in  \mS(R)$ have coprime factorizations 
$$
\mbfp_1=\displaystyle \frac{\mbfn_1}{\mbfd_1}, \quad \mbfp_2=\displaystyle \frac{\mbfn_2}{\mbfd_2}, \quad \mbfp_3=\displaystyle \frac{\mbfn_3}{\mbfd_3}.
$$
Since the usual Euclidean distance in $\mR^3$ satisfies the triangle inequality, it follows that 
\begin{eqnarray*}
%\displaystyle 
\frac{|\widehat{\mbfn_1}(\varphi)\widehat{\mbfd_2}(\varphi)-\widehat{\mbfn_2}(\varphi)\widehat{\mbfd_1}(\varphi)|}{
\sqrt{|\widehat{\mbfn_1}(\varphi)|^2+|\widehat{\mbfd_1}(\varphi)|^2} 
\sqrt{|\widehat{\mbfn_2}(\varphi)|^2+|\widehat{\mbfd_2}(\varphi)|^2}} 
%\!\!\!\!\!
&\leq&
%\displaystyle 
\frac{|\widehat{\mbfn_1}(\varphi)\widehat{\mbfd_3}(\varphi)-\widehat{\mbfn_3}(\varphi)\widehat{\mbfd_1}(\varphi)|}{
\sqrt{|\widehat{\mbfn_1}(\varphi)|^2+|\widehat{\mbfd_1}(\varphi)|^2} 
\sqrt{|\widehat{\mbfn_3}(\varphi)|^2+|\widehat{\mbfd_3}(\varphi)|^2}}
\\
&&+
%\displaystyle
\frac{|\widehat{\mbfn_3}(\varphi)\widehat{\mbfd_2}(\varphi)-\widehat{\mbfn_2}(\varphi)\widehat{\mbfd_3}(\varphi)|}{
\sqrt{|\widehat{\mbfn_3}(\varphi)|^2+|\widehat{\mbfd_3}(\varphi)|^2} 
\sqrt{|\widehat{\mbfn_2}(\varphi)|^2+|\widehat{\mbfd_2}(\varphi)|^2}}
\end{eqnarray*}
Consequently, $\kappa(\mbfp_1, \mbfp_2) \leq \kappa(\mbfp_1, \mbfp_2)+\kappa(\mbfp_1, \mbfp_2)$.
This completes the proof.
\end{proof}

Our main result is the following, which we will prove in the next section.

\begin{theorem}
 \label{main_theorem}
Suppose that $\mbfp_0, \mbfp \in \mS(R)$ and $\mbfc\in R$ is such that 
 $
\mbfg_0:= \displaystyle \frac{\mbfp_0}{1-\mbfc\mbfp_0} \in R.
$
Set $k:= \|\mbfc\|_\infty$ and $g=\|\mbfg_0\|_\infty$. If 
$$
\kappa(\mbfp,\mbfp_0)< \displaystyle \frac{1}{3} \min\left\{1,\;\; \frac{1}{g}, \;\;\frac{1}{k(1+kg)} \right\},
$$
then $\mbfp$ is also stabilized by $\mbfc$.
\end{theorem}

\section{Proof of the main result}

Lemma~\ref{lemma_partington} plays a key role in the proof of Theorem~\ref{main_theorem}, and so we 
include its short proof (taken from \cite[Lemma~2.1, p.84]{Par92}) here. 

\begin{proof}[Proof of Lemma~\ref{lemma_partington}] Consider the three possible cases, which are collectively exhaustive:
\begin{itemize}
 \item[$\underline{1}^\circ$] $|z_1| \leq \displaystyle \frac{1}{a}$ and $|z_2|\leq  \displaystyle\frac{1}{a}$. Then 
%$
%\sqrt{1+|z_k|^2}\leq \sqrt{1+\frac{1}{a^2}},\quad k=1,2,
%$
%and so 
$\kappa(z_1,z_2) \geq  \displaystyle\frac{a^2}{1+a^2} |z_1-z_2|$. 

\item[$\underline{2}^\circ$] $|z_1| \geq a$ and $|z_2|\geq  a$. 
Then $\displaystyle\frac{1}{|z_1|}\leq \displaystyle\frac{1}{a}$ and $ \displaystyle\frac{1}{|z_2|}\leq \displaystyle\frac{1}{a}$. 
As $\kappa(z_1,z_2)=\kappa\left( \displaystyle\frac{1}{z_1}, \displaystyle\frac{1}{z_2}\right)$, it follows from $\underline{1}^\circ$ above that 
$\kappa(z_1,z_2) \geq \displaystyle\frac{a^2}{1+a^2} \left|\frac{1}{z_1}-\frac{1}{z_2}\right|$. 

\item[$\underline{3}^\circ$] $|z_1| \leq a$ and $|z_2|\geq  \displaystyle\frac{1}{a}$, or vice versa. 
Since the distance between the spherical caps on the Riemann sphere corresponding to the regions 
$\{z\in \mC: |z|\leq a\}$ and $\left\{z\in \mC: |z| \geq \displaystyle\frac{1}{a}\right\}$ is 
$\kappa\left(a,\displaystyle\frac{1}{a}\right)= \displaystyle\frac{1-a^2}{1+a^2}$, it follows that 
$\kappa(z_1,z_2)\geq \displaystyle\frac{1-a^2}{1+a^2}$.  
\end{itemize}
This completes the proof.
\end{proof}

\begin{proof}[Proof of Theorem~\ref{main_theorem}] Let $\mbfp_0=\displaystyle \frac{\mbfn_0}{\mbfd_0}$ 
and $\mbfp=\displaystyle \frac{\mbfn}{\mbfd}$ be coprime factorizations of $\mbfp_0$ and $\mbfp$. 

Since $\mbfc$ stabilizes $\mbfp_0$, it follows in particular that 
$$
\frac{1}{1-\mbfp_0 \mbfc}=\frac{\mbfd_0}{\mbfd_0-\mbfn_0 \mbfc} \in R \textrm{ and } 
\frac{\mbfp_0}{1-\mbfp_0 \mbfc}=\frac{\mbfn_0}{\mbfd_0-\mbfn_0 \mbfc} \in R. 
$$
Moreover, since $(\mbfn_0,\mbfd_0)$ are coprime in $R$, there exist $\mbfx,\mbfy\in R$ such that 
$
\mbfn_0\cdot  \mbfx +\mbfd_0\cdot  \mbfy =1.
$
Hence it follows that 
$$
\frac{1}{\mbfd_0-\mbfn_0 \mbfc}=\frac{\mbfn_0\cdot  \mbfx +\mbfd_0\cdot  \mbfy}{\mbfd_0-\mbfn_0 \mbfc}
=\frac{\mbfp_0}{1-\mbfp_0 \mbfc} \cdot \mbfx +\frac{1}{1-\mbfp_0 \mbfc}\cdot \mbfy \in R.
$$
So $\mbfd_0-\mbfn_0 \mbfc$ is invertible as an element of $R$. In particular, it is also 
invertible as an element of $S$, and so 
\begin{equation}
 \label{eq_0}
\textrm{for all } \varphi \in M(S), \;\;\widehat{\mbfd_0}(\varphi)-\widehat{\mbfn_0}(\varphi) \widehat{\mbfc}(\varphi) \neq 0.
\end{equation}
Suppose that $\mbfd-\mbfn \mbfc $ is invertible as an element of $R$, then 
\begin{eqnarray*}
\frac{1}{1-\mbfp \mbfc}=\mbfd \cdot (\mbfd-\mbfn \mbfc )^{-1} \in R, &&
\frac{\mbfp}{1-\mbfp \mbfc}=\mbfn \cdot (\mbfd-\mbfn \mbfc )^{-1} \in R, \\
\frac{\mbfc}{1-\mbfp \mbfc}=\mbfc \cdot \mbfd \cdot (\mbfd-\mbfn \mbfc )^{-1} \in R, && 
\frac{\mbfp\mbfc}{1-\mbfp \mbfc}=-1+\mbfd \cdot (\mbfd-\mbfn \mbfc )^{-1} \in R, 
\end{eqnarray*}
and so $H(\mbfp, \mbfc) \in R^{2\times 2}$, showing that $\mbfp$ is also stabilized by $\mbfc$, and we are done.

So suppose that $\mbfd-\mbfn \mbfc $ is not invertible as an element of $R$. Then 
$\mbfd-\mbfn \mbfc $ is not invertible in $S$ too, since by assumption (A4), $R$ is a full 
subring of $S$. Thus there is a $\varphi_0\in M(S)$ such that 
\begin{equation}
 \label{eq_1}
\widehat{\mbfd}(\varphi_0)-\widehat{\mbfn}(\varphi_0) \widehat{\mbfc}(\varphi_0) =0.
\end{equation}

We consider the following cases. 

\bigskip 

\noindent $\underline{1}^\circ$ If $\widehat{\mbfd}(\varphi_0)=0$, then $\widehat{\mbfn}(\varphi_0)\neq 0$ by the 
coprimeness of $(\mbfd,\mbfn)$ and so by \eqref{eq_1}, $\widehat{\mbfc}(\varphi_0)=0$. 
Hence by \eqref{eq_0}, $\widehat{\mbfd_0}(\varphi_0)\neq 0$. 
So in this case we have 
$$
\kappa(\mbfp, \mbfp_0)\geq \frac{|\widehat{\mbfd_0}(\varphi_0)|}{\sqrt{|\widehat{\mbfn_0}(\varphi_0)|^2+|\widehat{\mbfd_0}(\varphi_0)|^2}}
= \kappa\left(\frac{\widehat{\mbfn_0}(\varphi_0)}{\widehat{\mbfd_0}(\varphi_0)}, \infty\right).
$$
But since $\widehat{\mbfc}(\varphi_0)=0$, we have 
$$
\left|\frac{\widehat{\mbfn_0}(\varphi_0)}{\widehat{\mbfd_0}(\varphi_0)}\right|
=
\left|\frac{\widehat{\mbfn_0}(\varphi_0)}{\widehat{\mbfd_0}(\varphi_0)-\widehat{\mbfn_0}(\varphi_0) \cdot \widehat{\mbfc}(\varphi_0) }\right|
=
|\widehat{\mbfg_0}(\varphi_0)|\leq \|\mbfg_0\|_\infty =g.
$$
Thus if $a$ is any number such that $0<a<1$, we have 
\begin{equation}
 \label{eq_2}
\kappa(\mbfp, \mbfp_0)\geq \kappa (g,\infty)=\kappa \left( \frac{1}{g}, 0\right)\geq \min \left\{\frac{a^2}{1+a^2}  \frac{1}{g},\;\;\frac{1-a^2}{1+a^2} \right\}.
\end{equation}

\bigskip 

\noindent $\underline{2}^\circ$ Now let $\widehat{\mbfd}(\varphi_0)\neq 0$. Then using \eqref{eq_1}, it follows that 
$\widehat{\mbfn}(\varphi_0)\neq 0$ and $\widehat{\mbfc}(\varphi_0)\neq 0$. 

Suppose first that $\widehat{\mbfd_0}(\varphi_0)=0$. By the coprimeness of 
$(\mbfd_0,\mbfn_0)$, we have $\widehat{\mbfn_0}(\varphi_0)\neq 0$. Then we have 
$$
\kappa(\mbfp, \mbfp_0)\geq \frac{|\widehat{\mbfd}(\varphi_0)|}{\sqrt{|\widehat{\mbfn}(\varphi_0)|^2+|\widehat{\mbfd}(\varphi_0)|^2}}
= \kappa\left(\frac{\widehat{\mbfn}(\varphi_0)}{\widehat{\mbfd}(\varphi_0)}, \infty\right)=
\kappa\left(\frac{1}{\widehat{\mbfc}(\varphi_0)}, \infty\right),
$$
where we have used \eqref{eq_1} to obtain the last equality. But 
$$
g= \|\mbfg_0\|_\infty =\sup_{\varphi \in M(S)} \left|\frac{\widehat{\mbfn_0}(\varphi)}{
\widehat{\mbfd_0}(\varphi)-\widehat{\mbfn_0}(\varphi)\widehat{\mbfc_0}(\varphi)}\right|
\geq \left|\frac{\widehat{\mbfn_0}(\varphi_0)}{
\widehat{\mbfd_0}(\varphi_0)-\widehat{\mbfn_0}(\varphi_0)\widehat{\mbfc_0}(\varphi_0)}\right|
=\frac{1}{|\widehat{\mbfc_0}(\varphi_0)|}.
$$
Thus if $a$ is any number such that $0<a<1$, we have 
\begin{equation}
 \label{eq_3}
\kappa(\mbfp, \mbfp_0)\geq\kappa\left(\frac{1}{\widehat{\mbfc}(\varphi_0)}, \infty\right) 
\geq \kappa (g,\infty)=\kappa \left( \frac{1}{g}, 0\right)\geq \min \left\{ \frac{a^2}{1+a^2} \frac{1}{g},\;\;\frac{1-a^2}{1+a^2} \right\}.
\end{equation}

Finally, suppose that $\widehat{\mbfd_0}(\varphi_0)\neq 0$. If $\widehat{\mbfn_0}(\varphi_0)=0$, then 
$$
\kappa(\mbfp, \mbfp_0)\geq \frac{|\widehat{\mbfn}(\varphi_0)|}{\sqrt{|\widehat{\mbfn}(\varphi_0)|^2+|\widehat{\mbfd}(\varphi_0)|^2}}
=\kappa\left(\frac{1}{\widehat{\mbfc}(\varphi_0)}, \infty\right),
$$
using \eqref{eq_1}, and proceeding in the same manner as above, we obtain \eqref{eq_3} once again. 

Suppose now that $\widehat{\mbfn_0}(\varphi_0)\neq 0$. We have  
$$
\kappa(\mbfp, \mbfp_0)\geq \kappa \left( \frac{\widehat{\mbfn}(\varphi_0)}{\widehat{\mbfd}(\varphi_0)}, 
\frac{\widehat{\mbfn_0}(\varphi_0)}{\widehat{\mbfd_0}(\varphi_0)} \right).
$$
Using \eqref{eq_1} we have that 
$$
\frac{\widehat{\mbfn}(\varphi_0)}{\widehat{\mbfd}(\varphi_0)}-\frac{\widehat{\mbfn_0}(\varphi_0)}{\widehat{\mbfd_0}(\varphi_0)}
=
\frac{1}{\widehat{\mbfc}(\varphi_0)}-\frac{\widehat{\mbfn_0}(\varphi_0)}{\widehat{\mbfd_0}(\varphi_0)}
=
\frac{1}{\widehat{\mbfc}(\varphi_0)}\left( 1-\widehat{\mbfc}(\varphi_0)\cdot \frac{\widehat{\mbfn_0}(\varphi_0)}{\widehat{\mbfd_0}(\varphi_0)}\right) .
$$
Clearly 
$$
\left|\frac{1}{\widehat{\mbfc}(\varphi_0)}\right|\geq \frac{1}{\|\mbfc\|_\infty} =\frac{1}{k}.
$$
Furthermore,
$$ 
\left|\frac{\widehat{\mbfd_0}(\varphi_0)}{\widehat{\mbfd_0}(\varphi_0)-\widehat{\mbfn_0}(\varphi_0)\widehat{\mbfc_0}(\varphi_0)}\right|
=
\left|1+\widehat{\mbfc_0}(\varphi_0)\frac{\widehat{\mbfn_0}(\varphi_0)}{\widehat{\mbfd_0}(\varphi_0)-\widehat{\mbfn_0}(\varphi_0)\widehat{\mbfc_0}(\varphi_0)}\right|
=
|1+\widehat{\mbfc_0}(\varphi_0)\widehat{\mbfg_0}(\varphi_0)|
\leq 
1+kg.
$$
Hence 
\begin{equation}
 \label{eq_4}
\left|\frac{\widehat{\mbfn}(\varphi_0)}{\widehat{\mbfd}(\varphi_0)}-\frac{\widehat{\mbfn_0}(\varphi_0)}{\widehat{\mbfd_0}(\varphi_0)}\right|
\geq \frac{1}{k(1+kg)}.
\end{equation}
Also, since $\widehat{\mbfn_0}(\varphi_0)\neq 0$, we have 
$$
\frac{\widehat{\mbfd}(\varphi_0)}{\widehat{\mbfn}(\varphi_0)}-\frac{\widehat{\mbfd_0}(\varphi_0)}{\widehat{\mbfn_0}(\varphi_0)}
=
\widehat{\mbfc}(\varphi_0)-\frac{\widehat{\mbfd_0}(\varphi_0)}{\widehat{\mbfn_0}(\varphi_0)}
=
-\frac{1}{\widehat{\mbfg_0}(\varphi_0)}.
$$
Thus 
\begin{equation}
 \label{eq_5}
\left|\frac{\widehat{\mbfd}(\varphi_0)}{\widehat{\mbfn}(\varphi_0)}-\frac{\widehat{\mbfd_0}(\varphi_0)}{\widehat{\mbfn_0}(\varphi_0)}\right| \geq 
\frac{1}{\|\mbfg_0\|_\infty} =\frac{1}{g}.
\end{equation}
Combining \eqref{eq_4} and \eqref{eq_5}, we obtain  that if $a$ is any number such that $0<a<1$, we have 
\begin{equation}
 \label{eq_6}
\kappa(\mbfp, \mbfp_0)\geq  \min \left\{ \frac{a^2}{1+a^2} \frac{1}{g},\;\;\frac{a^2}{1+a^2} \frac{1}{k(1+kg)}, \;\;\frac{1-a^2}{1+a^2} \right\}.
\end{equation}

Finally, \eqref{eq_2},\eqref{eq_3}, \eqref{eq_6} yield \eqref{eq_6} in {\em all} cases. With $a:=\displaystyle \frac{1}{\sqrt{2}}$, we obtain 
$$
\kappa(\mbfp, \mbfp_0)\geq  \frac{1}{3} \min  \left\{ \frac{1}{g},\;\; \frac{1}{k(1+kg)},\;\;1 \right\},
$$ 
which contradicts the hypothesis. Hence $\mbfd-\mbfn \mbfc$ is invertible as an element of $R$, 
and hence $\mbfp$ is stabilized by $\mbfc$. 
\end{proof}

\section{An example}

Consider the bidisc $\mD^2:=\mD\times \mD=\{(z_1,z_2)\in \mC^2: |z_1|< 1 \textrm{ and } |z_2|< 1\}$.
Let $R:=W^{1}(\mD^2)$ be the Wiener algebra of the bidisc, that is,
$$
W^1 (\mD^2):= \left\{f:= \sum_{k_1,k_2\geq 0} a_{k_1,k_2}z_1^{k_1} z_2^{k_2}: 
\|f\|_1:=\sum_{k_1,k_2\geq 0} |a_{k_1,k_2}|<+\infty\right\}.
$$
Then this is a relevant class of stable transfer functions arising in 
the analysis/synthesis of multidimensional digital filters, and membership in this class 
guarantees bounded input-bounded output (BIBO) stability; see for 
example \cite[\S2.1, p.3-4]{Bos}. 

Consider the nominal plant $\mbfp_0$ given by 
$$
\mbfp_0:=\frac{z_1 z_2}{z_1^2z_2^2-1} 
$$
which has the coprime factorization $\mbfp=\displaystyle \frac{\mbfn_0}{\mbfd_0}$, 
where 
$
\mbfn_0:=z_1 z_2, \; \mbfd_0:=z_1^2z_2^2-1.
$

%The pair $(\mbfn_0, \mbfd_0)$ is indeed coprime since with $\mbfx:=z_1z_2$ and 
%$\mbfy:=-1$, we have $\mbfn\cdot \mbfx+\mbfd \cdot \mbfy=1$.  
A stable controller which stabilizes $\mbfp_0$ is 
$\mbfc:=z_1z_2\in W^1(\mD^2)$, and we have 
$$
\mbfg_0:=\displaystyle \frac{\mbfp_0}{1-\mbfp_0\mbfc}=-z_1z_2.
$$
We take $S:=A(\mD^2)$, namely the bidisc algebra of functions continuous on $\overline{\mD}\times \overline{\mD}$ and 
holomorphic functions in $\mD^2$, with 
pointwise operations and the supremum norm:
$$
\|f\|_\infty:= \sup_{z_1,z_2\in \mD} |f(z_1,z_2)|, \quad f\in A (\mD^2).
$$
Then since the maximal ideal spaces of $W^1(\mD^2)$ and of $A(\mD^2)$ can both be identified with 
$\overline{\mD}\times  \overline{\mD}$ \cite[Theorem~11.7, p.279]{Rud}, it follows that $W^1(\mD^2)$ 
is full subalgebra in $A(\mD^2)$. 

Clearly, $g:=\|\mbfg_0\|_\infty =\|-z_1 z_2\|_\infty=1$ and $k:=\|\mbfc\|_\infty=\|z_1z_2\|_\infty=1$. 
So for all $\mbfp \in \mS(W^1(\mD^2))$ satisfying 
$$
\kappa(\mbfp , \mbfp_0) < \frac{1}{3} \min\left\{1,\;\; \frac{1}{g}, \;\;\frac{1}{k(1+kg)} \right\}=\frac{1}{3} 
\min\left\{1, \;\; \frac{1}{1(1+1\cdot 1)}, 
\;\; \frac{1}{1}\right\}=\frac{1}{6},
$$
$\mbfp$ is also stabilized by $\mbfc$. 
In particular, if we consider plants of the form 
$$
\mbfp_\alpha:= \frac{z_1 z_2-\alpha}{z_1^2z_2^2-1},
$$
for real $\alpha$ satisfying $|\alpha|<1$, then we can estimate $
\kappa(\mbfp_\alpha,\mbfp_0)$ as follows. We have 
\begin{eqnarray*}
\kappa(\mbfp_\alpha,\mbfp_0)&=&\sup_{z_1,z_2\in \mD} \frac{|\alpha| |z_1^2 z_2^2 -1|}{\sqrt{|z_1 z_1-\alpha|^2+|z_1^2 z_2^2 -1|^2}\sqrt{|z_1 z_1|^2+|z_1^2 z_2^2 -1|^2}}
\\
&\leq& 
\sup_{z_1,z_2\in \mD} \frac{|\alpha|}{\sqrt{|z_1 z_1|^2+|z_1^2 z_2^2 -1|^2}}
\leq
\sup_{0\leq k\leq 1}  \frac{|\alpha|}{\sqrt{k^2+(1-k^2)^2}}=\frac{2}{\sqrt{3}}|\alpha|.
\end{eqnarray*}
Thus for $\alpha$ satisfying $|\alpha|<\displaystyle\frac{1}{4\sqrt{3}}$,  $\mbfp_\alpha$ is stabilized by $\mbfc$. 

\medskip

  \noindent {\bf Acknowledgements:} The author thanks Jonathan Partington 
  for kindly providing a copy of \cite{Par92}, and Rudolf Rupp for useful 
  comments on a previous draft of the article.

\end{document}